\documentclass[12pt]{amsart}
\usepackage{amsmath,amsfonts,euscript,dsfont,amscd,amsthm,amssymb,upref,graphics}
\usepackage[all]{xy}
\theoremstyle{definition}

\swapnumbers

\theoremstyle{plain}

\newtheorem{theorem}{Theorem}[section]

\newtheorem{lemma}[theorem]{Lemma}

\theoremstyle{definition}

\newtheorem{definition}[theorem]{Definition}

\newtheorem{parag}[theorem]{}
\newtheorem{example}[theorem]{Example}

\newtheorem{notations}[theorem]{Notations}

\newtheorem{remarks}[theorem]{Remarks}

\theoremstyle{remark}

\newtheorem*{smallremark}{Remark}
\newenvironment{Enumerate}[1]%
{\begin{enumerate}\setlength{\itemsep}{#1}}{\end{enumerate}}
\newenvironment{enumerata}%
{\begin{enumerate}

}{\end{enumerate}}
{\begin{enumerata}\setlength{\itemsep}{#1}}{\end{enumerata}}
\newcommand{\Spec}{	\operatorname{{\rm Spec}}}

\newcommand{\supp}{	\operatorname{{\rm supp}}}
\newcommand{\image}{	\operatorname{{\rm im}}}

\newcommand{\Frac}{	\operatorname{{\rm Frac}}}
\newcommand{\Char}{	\operatorname{{\rm char}}}

\newcommand{\Div}{	\operatorname{{\rm Div}}}

\renewcommand{\div}{	\operatorname{{\rm div}}}

\newcommand{\bZ}{\mathbf{Z}}

\newcommand{\setspec}[2]{\big\{\,#1\, \mid \,#2\, \big\}}

\newcommand{\Integ}{\ensuremath{\mathbb{Z}}}
\newcommand{\Nat}{\ensuremath{\mathbb{N}}}

\newcommand{\Comp}{\ensuremath{\mathbb{C}}}

\newcommand{\aff}{\ensuremath{\mathbb{A}}}
\newcommand{\proj}{\ensuremath{\mathbb{P}}}
\newcommand{\bk}{{\ensuremath{\rm \bf k}}}

\newcommand{\kk}[1]{\bk^{[#1]}}

\newcommand{\mgoth}{{\ensuremath{\mathfrak{m}}}}
\newcommand{\ngoth}{{\ensuremath{\mathfrak{n}}}}

\newcommand{\Aeul}{\EuScript{A}}
\newcommand{\Beul}{\EuScript{B}}

\newcommand{\Geul}{\EuScript{G}}

\newcommand{\Oeul}{\EuScript{O}}

\newcommand{\Reul}{\EuScript{R}}
\newcommand{\Seul}{\EuScript{S}}

\newcommand{\isom}{\cong}
\renewcommand{\epsilon}{\varepsilon}
\renewcommand{\phi}{\varphi}
\renewcommand{\emptyset}{\varnothing}

\newcommand{\barr}{\overline}

\newlength{\mylength}
\newcommand{\rien}[1]{}

\newcommand{\abh}{%
\raisebox{-1.2mm}{%
\setlength{\unitlength}{1mm}%
\begin{picture}(4.4,4.4)(-2.2,-2.2)
\put(0,0){\circle{2}}
\put(-2,0){\line(1,0){4}}
\end{picture}}}

\CompileMatrices

\begin{document}
\renewcommand{\baselinestretch}{1.07}

%%%%%%	TOPMATTER:   %%%%%%%%%%%%%%%%%%%%%%%%%

\title[Generally rational polynomials]{Generally rational polynomials in two variables}

\author{Daniel Daigle}

% \date{\today}

\address{Department of Mathematics and Statistics\\ University of Ottawa\\ Ottawa, Canada\ \ K1N 6N5}
\email{ddaigle@uottawa.ca}

\thanks{Research supported by grant RGPIN/104976-2010 from NSERC Canada.}

{\renewcommand{\thefootnote}{}
\footnotetext{2010 \textit{Mathematics Subject Classification.}
Primary: 14R10.  Secondary: 14M20, 13F20, 14G17.}}

\keywords{Polynomial ring, rational curve,
field generator, generally rational polynomial, generically rational polynomial,
pencil, positive characteristic, purely inseparable extension.}

\begin{abstract} 
Let $\bk$ be an algebraically closed field.
A polynomial $F \in \bk[X,Y]$ is said to be {\it generally rational\/} if,
for almost all $\lambda \in \bk$, the curve ``$F=\lambda$'' is rational.
It is well known that, if  $\Char\bk=0$, $F$ is generally rational
iff there exists $G \in \bk(X,Y)$ such that $\bk(F,G)=\bk(X,Y)$.
We give analogous results valid in arbitrary characteristic.
\end{abstract}

\maketitle
  
\vfuzz=2pt

%%%%%%%%%%%%%%%%%%%%%%%%%%%%%%%%%%%%%%%%%%%%%%%%%%%%%%%%%%%%%%%%%%%
%%%%%%%%%%%%%%%%%%%%%%%%%%%%%%%%%%%%%%%%%%%%%%%%%%%%%%%%%%%%%%%%%%%
%%%%%%%%%%%%%%%%%%%%%%%%%%%%%%%%%%%%%%%%%%%%%%%%%%%%%%%%%%%%%%%%%%%

\section{Definitions and statements of results}
\label {Sec:Definitionsandstatementsofresults}

Given rings $R \subseteq S$, we write $S = R^{[n]}$ to indicate that $S$ is isomorphic, as an $R$-algebra,
to the polynomial algebra in $n$ variables over $R$.  If $L/K$ is a field extension, 
we write $L = K^{(n)}$ to indicate that $L$ is a purely transcendental extension of $K$, of transcendence 
degree $n$.  The field of fractions of a domain $R$ is denoted $\Frac R$.

\begin{definition}  \label {fp939223ekfjowe8}
Let $\bk$ be a field and $F \in A = \kk2$.
\begin{Enumerate}{1mm}

\item We define the phrase ``$A/(F)$ is $\bk$-rational'' to mean:
$$
\text{$F$ is an irreducible element of $A$ and the field of fractions of $A/(F)$ is $\bk^{(1)}$.}
$$

\item Suppose that $\bk$ is algebraically closed.
We say that $F$ is a {\it generally rational polynomial in $A$} if 
$A/(F-\lambda)$ is $\bk$-rational for almost all $\lambda \in \bk$,
where by ``almost all'' we mean ``all except possibly finitely many''.

\end{Enumerate}
\end{definition}

\begin{smallremark}
In \ref{o83020192djwoidja}, below, we show that if
$A/(F-\lambda)$ is $\bk$-rational for infinitely many $\lambda \in \bk$
then it is $\bk$-rational for almost all $\lambda \in \bk$.
\end{smallremark}

\begin{smallremark}
In the literature, generally rational polynomials are sometimes called ``generically rational polynomials'' or simply
``rational polynomials''.  The term ``generically rational polynomial'' is particularly misleading
since it suggests that the fiber of $\Spec A \to \Spec\bk[F]$ over the generic point of $\Spec\bk[F]$
is rational, which is not the intended meaning. (Note that the fiber over the generic point is rational
if and only if $F$ is a field generator, cf.\ \ref{hj24378rduq27}.)
\end{smallremark}

\begin{definition} \label {hj24378rduq27}
Let $\bk$ be a field and $F \in A = \kk2$.
We say that $F$ is a {\it field generator in $A$} if there exists $G \in \Frac A$ such
that $\bk(F,G)=\Frac A$.
If $G$ can be chosen in $A$, we say that $F$ is a {\it good\/} field generator in $A$;
if not, we say that $F$ is {\it bad}.
(Cf.\ \cite{JanThesis}, \cite{Rus:fg}, \cite{Rus:fg2}, \cite{Cassou-BadFG}.)
\end{definition}

It is known that if $\bk$ is an algebraically closed field of characteristic zero,
then $F \in \bk[X,Y]$ is a field generator if and only if it is a generally rational polynomial
(this is mentioned, for instance, in the introduction of \cite{MiySugie:GenRatPolys}).
In positive characteristic, one knows examples of generally rational polynomials which are
not field generators, but, apparently,
the precise relation between the two notions remains to be clarified.
It is the aim of the present paper to provide such clarification.
In order to do so, we propose the following

\begin{definition}
Let $\bk$ be a field and $F \in A = \kk2$.
We say that $F$ is a {\it pseudo field generator (PFG) in $A$} if there exists $G \in \Frac A$ such
that $\Frac A$ is a purely inseparable extension of $\bk(F,G)$.
If $G$ can be chosen in $A$, we say that $F$ is a {\it good\/} pseudo field generator in $A$;
if not, we say that $F$ is {\it bad}.
\end{definition}

\begin{remarks}
Let $\bk$ be a field and $F \in A = \kk2$.
\begin{enumerate}

\item
It is clear that ``field generator'' implies ``pseudo field generator'', and that the two notions are equivalent
if $\Char\bk=0$.

\item If $\Char\bk=p>0$ then the following hold:
\begin{itemize}
\item $F$ is a PFG in $A$ iff $F^p$ is a PFG in $A$.
\item $F$ is a good PFG in $A$ iff $F^p$ is a good PFG in $A$.
\end{itemize}

\end{enumerate}
\end{remarks}

Our aim is to prove
Theorems~\ref{o9823ry91c83yrquwei},  \ref{q9830732847ryt23rhf}, \ref{9230r9984tyqrawdc9r0fhd}
and \ref{020399f09fjcnEH1283} (the proofs are given in Section \ref{Sec:Proofs}).
Throughout, our base field is algebraically closed and of arbitrary characteristic.
Our results are well known in the case $\Char\bk=0$.
In fact, we recover the case $\Char\bk=0$ as a special case of our results.

\begin{theorem} \label {o9823ry91c83yrquwei}
Let $\bk$ be an algebraically closed field and let $A=\kk2$.
For $F \in A$, the following conditions are equivalent:
\begin{enumerata}

\item \label {982373253hhkj947}
 $F$ is a generally rational polynomial in $A$;

\item  \label {989kjbgcdsaw4567u9jdf}
$F$ is a pseudo field generator in $A$ and if $\Char\bk=p>0$ then $F \notin A^p$.

\end{enumerata}
\end{theorem}

\begin{definition}  \label {ij920201e9d2jxd}
Let $\bk$ be an algebraically closed field and let $A=\kk2$.
Consider $F \in A \setminus \bk$ such that, for almost all $\lambda \in \bk$, $F-\lambda$ is irreducible in $A$.
Let $f: \aff^2= \Spec A \to  \aff^1=\Spec\bk[F]$ be the morphism determined by the inclusion $\bk[F] \hookrightarrow A$.
Choose a commutative diagram
\begin{equation}  \label {7y74r47r77392392093jeij}
\xymatrix{
X \ar[r]^{\bar f}  &  \proj^1  \\
\aff^2  \ar @{^{(}->} [u] \ar[r]_{f}  &  \aff^1 \ar @{^{(}->} [u]
}
\end{equation}
where $X$ is a nonsingular projective surface, the vertical arrows are open immersions, and $\bar f$ is a morphism.
Note that $\bar f^{-1}(P)$  is an integral curve for almost all closed points $P \in \proj^1$.
\begin{enumerate}

\item We say that {\it $(F,A)$ has no moving singularities\/} if
$\bar f^{-1}(P)$ is a nonsingular curve for almost all closed points $P \in \proj^1$.

\item We say that {\it $(F,A)$ has no moving singularities at finite distance\/} if
$f^{-1}(P)$ is a nonsingular curve for almost all closed points $P \in \aff^1$.

\end{enumerate}
These properties depend only on $(F,A)$, i.e., are independent of the choice of diagram~\eqref{7y74r47r77392392093jeij}.
\end{definition}

\begin{remarks}
Let the assumptions on $\bk$, $A$, $F$ be as in \ref{ij920201e9d2jxd},
and consider the question whether $(F,A)$ has moving singularities.
\begin{enumerate}

\item If $\Char\bk=0$ then $(F,A)$ has no moving singularities, by a theorem of Bertini.

\item Assume that $\Char\bk=p>0$.
If $(F,A)$ has no moving singularities then it has no moving singularities at finite distance.
However the converse is not true (see \ref{pfupq293up9qwjd}, for instance).
\end{enumerate}
\end{remarks}

\begin{theorem} \label {q9830732847ryt23rhf}
Let $\bk$ be an algebraically closed field and let $A=\kk2$.
For $F \in A$, the following conditions are equivalent:
\begin{enumerata}
\item \label {k3kk34k2j3ki3jo3i4uy7}
$F$ is a generally rational polynomial in $A$ and $(F,A)$ has no moving singularities;
\item \label {8788sx88x8x8x878x7x8x7f} $F$ is a field generator in $A$.
\end{enumerata}
\end{theorem}

Given a field extension $F/E$, a valuation ring ``of $F$ over $E$'' is a valuation ring $\Oeul$
satisfying $E \subseteq \Oeul \subseteq F$ and $\Frac \Oeul = F$.

\begin{definition} \label {0293cn2r929af498}
\begin{enumerate}

\item Given $E \subset B$, where $E$ is a field and $B$ is a domain,
we write $\proj(B/E)$ for the set of all valuation rings $\Oeul$ of $\Frac B$ over $E$ satisfying $\Oeul \neq \Frac B$;
we also set 
$$
\proj_\infty(B/E) = \setspec{ \Oeul \in \proj(B/E) }{ B \nsubseteq \Oeul }
\text{\ \ and\ \ }
\proj_{\text{\rm fin}}(B/E) = \setspec{ \Oeul \in \proj(B/E) }{ B \subseteq \Oeul } .
$$
The elements of $\proj_\infty(B/E)$ are called the ``places at infinity'' of $B/E$.
Note that $\bigcap \proj_{\text{\rm fin}}(B/E)$ is the integral closure of $B$ in $\Frac B$.\footnote{%
We abbreviate $\bigcap_{ \Oeul \in \proj_{\text{\rm fin}}(B/E)}\!\Oeul$ to
$\bigcap \proj_{\text{\rm fin}}(B/E)$
and we decree that $\bigcap \proj_{\text{\rm fin}}(B/E) = \Frac B$
when $\proj_{\text{\rm fin}}(B/E) = \emptyset$.}

\item Let $\bk$ be a field, $F$ an irreducible element of $A = \kk2$ and $R = A/(F)$.
Then it is customary to refer to the elements of 
$\proj_\infty(R/\bk)$ as the places at infinity of $R$, or of $\Spec R$, or of $F$.
The cardinal number $| \proj_\infty(R/\bk) |$ is a positive integer; if it is $1$, we say
that $R$ (or $\Spec R$, or $F$) has one place at infinity.

\item Let $\bk$ be a field and $F \in A = \kk2$, $F \notin \bk$.
Let $\Aeul = S^{-1}A$ where $S = \bk[F] \setminus \{0\}$. Then the elements of 
$\proj_\infty( \Aeul/\bk(F) )$ are called the {\it dicriticals\/} of $F$ (or more correctly, of the pair $(F,A)$).
Given a dicritical $\Oeul \in \proj_\infty( \Aeul/\bk(F) )$,
the residue field $\kappa$ of $\Oeul$ is a finite extension of $\bk(F)$; the number $[\kappa:\bk(F)]$ is called the 
{\it degree\/} of the dicritical;
one says that the dicritical $\Oeul$ is {\it purely inseparable\/} if $\kappa$ is
purely inseparable over $\bk(F)$.
Note that a dicritical of $F$ is the same thing as a place at infinity of $\Aeul/\bk(F)$.
By ``the number of dicriticals of $F$'' we mean the cardinal number $| \proj_\infty(\Aeul/\bk(F)) |$,
which is a positive integer.

\end{enumerate}
\end{definition}

In \cite[Rem.\ after 1.3]{Rus:fg}, 
Russell observes that a field generator $F \in A$ is good if and only if it has at least one 
dicritical of degree $1$.
The next result gives an analogous criterion for pseudo field generators.

\begin{theorem} \label {9230r9984tyqrawdc9r0fhd}
Let $\bk$ be an algebraically closed field and let $F \in A=\kk2$ be a pseudo field generator in $A$.
The following conditions are equivalent:
\begin{enumerata}
\item \label {928fhnc1bci3urh} $F$ is a good pseudo field generator in $A$;
\item \label {432kln9c745jdie} $F$ has at least one purely inseparable dicritical.
\end{enumerata}
\end{theorem}

The case $\bk=\Comp$ of the next result can be found in \cite[Th.\ 2]{Suzuki} and 
\cite[Cor.\ 2]{Kal:TwoRems};
the more general case $\Char\bk=0$ is proved in \cite[1.6]{MiySugie:GenRatPolys}.
The case $\Char\bk>0$ appears to be new.

\begin{theorem}  \label {020399f09fjcnEH1283}
Let $\bk$ be an algebraically closed field and let $F \in A = \kk2$ be a generally rational polynomial of $A$.
Then
$$
t-1 = \sum_{ \lambda \in \bk } (n_\lambda - 1)
$$
where $t$ is the number of dicriticals of $F$ and $n_\lambda$ is the number of 
irreducible components of the closed subset $V(F-\lambda)$ of $\Spec A$.
\end{theorem}

\section*{Remarks and examples}

It is quite clear that Theorems \ref{o9823ry91c83yrquwei} and \ref{q9830732847ryt23rhf}
are of the same nature:
each states the equivalence of two conditions on $F \in A$,
the first being a property of the fiber of $F$ over a general closed point, and
the second, an algebraic property of the pair $(F,A)$ which is a weakening
of the condition ``there exists $G$ satisfying $A = \bk[F,G]$''.

To gain some perspective, we shall now recall two more results of the same type
(\ref{2930923i2indedp} and \ref{o902939jcoidjq239}).
One could formulate these facts in a characteristic-free language,
as we did in \ref{fp939223ekfjowe8}--\ref{020399f09fjcnEH1283}, 
but for the sake of simplicity we mainly consider the case $\Char\bk>0$  in this discussion.

\medskip
\noindent{\bf Polynomial curves.}
Let $\bk$ be an algebraically closed field.
An affine curve over $\bk$ is called a {\it polynomial curve\/} if it is rational and has one place at
infinity.
Abusing language, one says that an irreducible $F \in A = \kk2$ is a ``polynomial curve in $A$''
if $\Spec A/(F)$ is a polynomial curve.\footnote{Apparently, the term ``polynomial curve'' was coined by Abhyankar.
Note that $F$ is a polynomial curve in $A=\kk2$ if and only if $A/(F)$ is a subalgebra of a $\kk1$.
That is, a polynomial curve is an affine curve that can be parametrized by univariate polynomials.}
The first result that we want to recall is:

\begin{theorem}[\cite{Dai:pencils}] \label {2930923i2indedp}
Let $A = \kk2$, where $\bk$ is algebraically closed and of characteristic $p>0$.
For $F \in A$, the following are equivalent:
\begin{enumerate}

\item for almost all $\lambda \in \bk$, $F-\lambda$ is a polynomial curve in $A$;

\item $F \notin A^p$ and there exist $G \in A$ and $n \in \Nat$ such
that $A^{p^n} \subseteq \bk[F,G]$. 

\end{enumerate}
\end{theorem}

Theorem~\ref{2930923i2indedp} is a corollary of the main result of \cite{Dai:pencils}.
In that paper, one says that $F \in A$ is a {\it $p$-generator in $A$} if
there exist $G \in A$ and $n\ge0$ such that $A^{p^n} \subseteq \bk[F,G]$
(so condition (2) of \ref{2930923i2indedp} states that $F$ is a $p$-generator in $A$ which does not
belong to $A^p$).
Clearly, every $p$-generator in $A$ is a good PFG in $A$ (the converse is not true, by \ref{ai93948rp9qjewpqjra}).
Also note that $F$ is a  $p$-generator in $A$ iff $F^p$ is.

\medskip
\noindent{\bf Lines.}
Let $\bk$ be a field 
and $F \in A = \kk2$.
If there exits $G$ such that $A = \bk[F,G]$, one says that $F$ is a {\it variable in $A$}.
If $A/(F) = \kk1$, one says that $F$ is a {\it line in $A$}.
Obviously, every variable is a line; a line which is not a variable is called an {\it exotic line}.
The Abhyankar-Moh-Suzuki Theorem (\cite{A-M:line}, or \cite{Suzuki} if $\bk=\Comp$)
implies that exotic lines do not exist if $\Char\bk=0$.
If $\bk$ is any field of characteristic $p>0$, then 
$F=X^{p^2}+Y^{p(p+1)}+Y$
is an example of an exotic line in $\bk[X,Y]$.

The second (and last) result that we want to recall is:

\begin{theorem}[\cite{Gan:thesis}] \label {o902939jcoidjq239}
Let $A = \kk2$, where $\bk$ is algebraically closed and of characteristic $p>0$.
For $F \in A$, the following are equivalent:
\begin{enumerate}

\item \label {352r2gyrhilpo987d62} $F-\lambda$ is a line in $A$, for all $\lambda \in \bk$;
\item $F-\lambda$ is a line in $A$, for almost all $\lambda \in \bk$;
\item $F \notin A^p$ and there exist $n \in \Nat$ and $G \in A$ such that $A^{p^n}[F] = \bk[F,G]$.

\end{enumerate}
\end{theorem}

(This is a consequence of either one of \cite[3.1 and 4.12]{Ganong:Survey} or \cite[3.13 and 3.14]{Gan:thesis};
more equivalent conditions are given in \cite{Ganong:Survey}, \cite{Gan:thesis}.)

\smallskip
It is obvious that if $F \in A$ satisfies the  equivalent conditions of \ref{o902939jcoidjq239} then $F$
is a line in $A$.  The converse, however, is an open question. 
It is clear that if $F$ is a variable in $A$ then $F$ satisfies those conditions,
and all currently known examples of exotic lines in $A$ also satisfy them,
but it is not known whether all exotic lines have that property.
See \cite{Ganong:Survey} for a discussion of this question.

\begin{parag} \label {zz1jp293r0912wdweoia888}
To summarize, consider the following four subsets of $A = \kk2$
(where $\bk$ is an algebraically closed field of characteristic $p>0$):
\begin{itemize}

\item[] \hspace*{-9mm}
$E_1=$ the set of generally rational polynomials in $A$,
which is equal (by \ref{o9823ry91c83yrquwei}) to the set of PFGs in $A$ not belonging to $A^p$;

\item[] \hspace*{-1cm}
$E_2=$ the set of generally rational polynomials $F$ in $A$ such that $(F,A)$ has no moving singularities,
which is equal (by \ref{q9830732847ryt23rhf}) to the set of field generators in $A$;

\item[] \hspace*{-1cm}
$E_3=$ the set of $F \in A$ such that $F-\lambda$ is a polynomial curve in $A$ for almost all $\lambda \in \bk$,
which is equal (by \ref{2930923i2indedp}) to the set of $p$-generators in $A$ not belonging to $A^p$;

\item[] \hspace*{-1cm}
$E_4=$ the set of $F \in A$ such that $F-\lambda$ is a line in $A$ for almost all $\lambda \in \bk$,
which is equal (by \ref{o902939jcoidjq239}) to the set of $F \in A$ satisfying $F \notin A^p$ and
$\exists_{G,n}$ $A^{p^n}[F]= \bk[F,G]$.

\end{itemize}
Then the following hold:
\begin{enumerate}
\item[(i)] \label {pfoup23909wd}
\text{$E_2 \subset E_1 \supset E_3 \supset E_4$, where all inclusions are strict;}
\item[(ii)] \label {87676x7dx9xc778vx78cx6}
\text{$E_2 \cap E_3 = E_2 \cap E_4 =$ the set of variables of $A$.}
\end{enumerate}
Indeed, inclusions $E_2 \subseteq E_1 \supseteq E_3$ are obvious,
and $E_3 \supseteq E_4$ holds because every line is a rational curve with one place at infinity;
all inclusions are strict by examples \ref{pfp32w9rp9qjfaej} and \ref{ai93948rp9qjewpqjra}.
Assertion (ii) follows from the fact (cf.\ \cite[4.5]{Rus:fg})
that any field generator which has one place at infinity is in fact a variable.
\end{parag}

In the following examples, 
we let $A = \bk[X,Y] = \kk2$
where $\bk$ is algebraically closed and of characteristic $p>0$.

\begin{example}  \label {pfupq293up9qwjd}
Let $F \in A = \bk[X,Y]$ be any exotic line satisfying the  equivalent conditions of~\ref{o902939jcoidjq239}
(for instance, $F=X^{p^2}+Y^{p(p+1)}+Y$) and let $f : \aff^2 \to \aff^1$ be the morphism determined by
the inclusion $\bk[F] \hookrightarrow A$.  By \ref{o902939jcoidjq239}\eqref{352r2gyrhilpo987d62},
$f^{-1}(P) \isom \aff^1$ for every closed point $P \in \aff^1$; 
in particular,
\begin{enumerate}
\item[(i)] $F$ is a generally rational polynomial in $A$ and $(F,A)$ has no moving singularities at finite distance.
\end{enumerate}
As was mentioned in \ref{zz1jp293r0912wdweoia888}, any field generator which has one place at infinity is a variable.
As lines have one place at infinity, it follows that no exotic line is a field generator.
So:
\begin{enumerate}
\item[(ii)] $F$ is not a field  generator in $A$.
\end{enumerate}
The reader should compare (i, ii) to the statement of \ref{q9830732847ryt23rhf}.
Note in particular that $(F,A)$ has moving singularities, but not at finite distance.
\end{example}

\begin{example} \label {pfp32w9rp9qjfaej}
Let $F=X^p+Y^{p+1} \in A = \bk[X,Y]$.
Then $A^p \subset \bk[F,Y]$, so $F$ is a $p$-generator (hence a good PFG) in $A$.
For every $\lambda \in \bk$, $A/(F-\lambda)$ is a singular $\bk$-rational curve with one place at infinity.
By \ref{q9830732847ryt23rhf}, $F$ is not a field generator in $A$.
\end{example}

\begin{example} \label {ai93948rp9qjewpqjra}
Let $F=(X^p+Y^{p+1})Y \in A = \bk[X,Y]$.
Then $\bk(X,Y)$ is purely inseparable over $\bk(F,Y)=\bk(X^p,Y)$, so $F$ is a good PFG in $A$.
For almost all $\lambda \in \bk$, the $\bk$-curve $A/(F-\lambda)$ is a singular rational curve with two places at infinity.
By \ref{2930923i2indedp}, $F$ is not a $p$-generator in $A$;
by  \ref{q9830732847ryt23rhf},  it is not a field generator in $A$.
\end{example}

\section{Preliminaries to the proofs}

\begin{lemma} \label {Matsudp23rwije}
Let $\bk$ be an algebraically closed field and $F \in A = \kk n$. Then the set
$$
\setspec{ \lambda \in \bk }{ \text{$F-\lambda$ is not irreducible in $A$} }
$$
is either finite or equal to $\bk$, and it is equal to $\bk$ if and only if
$F = P(G)$ for some $G \in A$ and some univariate polynomial $P(T) \in \bk[T]$
such that $\deg_T P(T) > 1$.
\end{lemma}

\begin{proof}
This can be derived from a general Theorem on linear systems proved by Bertini
(and reproved by Zariski) in characteristic zero,
then generalized to all characteristics by Matsusaka~\cite{Matsusaka:Bertini1950}.
For the result as stated here, see \cite{SchinzelBook2000}, Chap.~3, \S~3, Cor.~1.
\end{proof}

\begin{notations} \label {q9nc29991209fh27gdaj}
Let $F/E$ be a function field in one variable and recall from \ref{0293cn2r929af498} that $\proj(F/E)$
is the set of valuation rings $\Oeul$ of $F$ over $E$ satisfying $\Oeul \neq F$.
The divisor group $\Div(F/E)$ is the free abelian group on the set $\proj(F/E)$;
given $\xi \in F^*$, we write $\div(\xi), \div_0(\xi), \div_\infty(\xi) \in \Div(F/E)$ for the principal divisor,
divisor of zeroes and divisor of poles of $\xi$, respectively.
\end{notations}

\begin{lemma}  \label {uhh6e34qfdpc034}
Let $\bk$ be a field and consider an irreducible $F(X,Y) \in A = \bk[X,Y] = \kk2$.
Then $A/(F)$ is $\bk$-rational if and only if there exists $(x(T), y(T), z(T) ) \in \bk[T]^3$ satisfying:
\begin{enumerate}

\item  $z(T) \neq 0$, $\big\{ \frac{x(T)}{z(T)},  \frac{y(T)}{z(T)} \big\} \nsubseteq \bk$ 
and $F( \frac{x(T)}{z(T)},  \frac{y(T)}{z(T)} ) = 0$;

\item $\max( \deg_T x(T), \deg_T y(T), \deg_T z(T) ) \le \deg_XF+\deg_YF$.

\end{enumerate}
\end{lemma}

\begin{proof}
It is clear that if $(x(T), y(T), z(T) )$ exists then $A/(F)$ is $\bk$-rational.
Conversely, suppose that $A/(F)$ is $\bk$-rational.
Then there exist $\phi,\psi \in \bk(T)=\bk^{(1)}$ satisfying 
$F(\phi,\psi)=0$ and $\bk(\phi, \psi) = \bk(T)$.
If $\phi \in \bk$ then\footnote{We use Abhyankar's symbol ``$\abh$'' to denote
an arbitrary nonzero element of the base field $\bk$.}
$F = \abh (X-a)$ (some $a \in \bk$) and $(x,y,z) = (a,T,1)$ satisfies the desired conditions.
Similarly, if $\psi \in \bk$  then $(x,y,z)$ exists.
From now-on, assume that $\phi,\psi \notin \bk$.
Considering divisors in $\Div( \bk(T)/\bk )$ with notation as in \ref{q9nc29991209fh27gdaj},
\begin{equation} \label {9r92891c2989812hed}
 \deg \div_0(\phi) = [\bk(T) : \bk(\phi)] =\deg_YF 
\quad \text{and} \quad
 \deg \div_0(\psi) = [\bk(T) : \bk(\psi)] =\deg_XF .
\end{equation}
Write $\phi=u/w_1$, $\psi=v/w_2$ where 
$u,v,w_1,w_2 \in \bk[T]$, $w_1, w_2 \neq 0$ and $\gcd(u,w_1)=1=\gcd(v,w_2)$.
Let $u = \abh\prod_{i=1}^m p_i^{e_i}$, $w_1 = \abh\prod_{i=1}^n q_i^{f_i}$ be the prime factorizations
of $u$ and $w_1$ respectively, where  $e_i,f_i>0$ and where the $p_i, q_i \in \bk[T]$ are $m+n$ distinct
monic irreducible polynomials.
Define 
$P_i = \bk[T]_{(p_i)}$ ($1\le i \le m$), $Q_i = \bk[T]_{(q_i)}$ ($1\le i \le n$) and $P_\infty = \bk[T^{-1}]_{(T^{-1})}$;
then $P_i,Q_i,P_\infty \in \proj( \bk(T)/\bk )$ and
$\div(\phi) = \sum_{i=1}^m e_iP_i - \sum_{i=1}^n f_i Q_i + (\deg w_1 - \deg u)P_\infty$, so:
\begin{itemize}

\item if $\deg w_1 > \deg u$ then
$\div_0(\phi) = \sum_{i=1}^m e_iP_i + (\deg w_1 - \deg u)P_\infty$ has degree equal to $\deg w_1$;

\item if $\deg w_1 \le \deg u$ then
$\div_0(\phi) = \sum_{i=1}^m e_iP_i$ has degree equal to $\deg u$;

\end{itemize}
so $\deg \div_0(\phi) = \max( \deg u, \deg w_1)$ in both cases. Then
$\max( \deg u, \deg w_1) = \deg_YF$ by \eqref{9r92891c2989812hed} and,
for similar reasons, $\max( \deg v, \deg w_2) = \deg_XF$.

So $(x,y,z) = ( uw_2, vw_1, w_1w_2 )$ satisfies the desired conditions.
\end{proof}

\begin{lemma}  \label {o83020192djwoidja}
Let $\bk$ be an algebraically closed field and $F \in A = \kk2$.
The following conditions are equivalent:
\begin{enumerate}
\item $A/(F-\lambda)$ is $\bk$-rational for infinitely many $\lambda \in \bk$;
\item $A/(F-\lambda)$ is $\bk$-rational for almost all $\lambda \in \bk$.
\end{enumerate}
\end{lemma}

\begin{proof}
Assume that (1) holds.
In particular, there exists $\lambda \in \bk$ such that $F-\lambda$ is irreducible in $A$;
then, by \ref{Matsudp23rwije}, $F-\lambda$ is irreducible in $A$ for almost all $\lambda \in \bk$.
Choose $X,Y$ such that $A=\bk[X,Y]$, let $d = \deg F$ and $n= \deg_XF + \deg_YF$,
and consider the homogenization $F^*(X,Y,Z) \in \bk[X,Y,Z]$ of $F$, i.e., $F^*(X,Y,Z) = Z^d F(X/Z, Y/Z)$.
Let
$
R = \bk[X_0, \dots, X_n, Y_0, \dots, Y_n, Z_0, \dots, Z_n, L] = \kk{3n+4}
$
and define $g_0, \dots, g_{nd} \in R$ by
$$
\textstyle
F^* \big( \sum_{i=0}^n X_iT^i, \sum_{i=0}^n Y_iT^i, \sum_{i=0}^n Z_iT^i  \big) - L \big( \sum_{i=0}^n Z_iT^i  \big)^d
= \sum_{i=0}^{nd} g_i T^i .
$$
Define ideals $I$ and $J$ of $R$ by stipulating that 
$I$ is generated by $g_0, \dots, g_{nd}$ and that
$J$ is generated by all $2 \times 2$ determinants
$\left| \begin{smallmatrix} X_i & X_j \\ Z_i & Z_j \end{smallmatrix} \right|$
and $\left| \begin{smallmatrix} Y_i & Y_j \\ Z_i & Z_j \end{smallmatrix} \right|$
with $0 \le i < j \le n$.
Consider the zero-sets $\bZ(I), \bZ(J) \subseteq \bk^{3n+4}$ of $I$ and $J$ respectively,
the locally closed subset $U = \bZ(I) \setminus \bZ(J)$ of $\bk^{3n+4}$
and the map $h : U \to \bk$ which is the restriction of the projection  $\bk^{3n+4} \to \bk$ on the last factor.

For $\lambda \in \bk$, the following are equivalent:
\begin{enumerate}

\item[(i)]  $\lambda \in \image h$;

\item[(ii)] there exist $(a_0, \dots, a_n), (b_0, \dots, b_n), (c_0, \dots, c_n) \in \bk^{n+1}$ such that,
if we define $x=\sum_{i=0}^n a_i T^i$, $y=\sum_{i=0}^n b_i T^i$ and $z=\sum_{i=0}^n c_i T^i$,
then $F^*(x,y,z) - \lambda z^d = 0$, $z \neq 0$ and 
$\big\{ \frac{x}{z},  \frac{y}{z} \big\} \nsubseteq \bk$;

\item[(iii)] there exist $(x,y,z) \in \bk[T]^3$ such that $\max( \deg_T x, \deg_T y, \deg_T z ) \le n$, 
$z \neq 0$, $\big\{ \frac{x}{z},  \frac{y}{z} \big\} \nsubseteq \bk$ and $F( \frac{x}{z},  \frac{y}{z} ) = \lambda$.

\end{enumerate}

Moreover, under the assumption that $F-\lambda$ is irreducible in $A$, \ref{uhh6e34qfdpc034} shows
that (iii) is equivalent to $A/(F-\lambda)$ being $\bk$-rational.

Since we assumed that (1) holds, $\image h$ is an infinite set.
As $\image h$ is a constructible subset of $\bk$, we obtain that $\bk \setminus \image h$ is a finite set.
Since $F-\lambda$ is irreducible for almost all $\lambda \in \bk$, (2) holds.
The converse is trivial.
\end{proof}

\begin{lemma} \label {f9032978y9187x4yriwedj}
Let $K \subseteq L$ be algebraically closed fields,
$X,Y$ indeterminates over $L$
and $F \in K[X,Y] \subseteq L[X,Y]$, $F \notin K$.  Then
\begin{enumerata}

\item $F$ is irreducible in $K[X,Y]$ $\iff$ $F$ is irreducible in $L[X,Y]$.

\item $K[X,Y]/(F)$ is $K$-rational $\iff$ $L[X,Y]/(F)$ is $L$-rational.

\item $F$ is a generally rational polynomial in $K[X,Y]$ \\
\mbox{\ } \hfill $\iff$ $F$ is a generally rational polynomial in $L[X,Y]$.

\end{enumerata}
\end{lemma}

\begin{proof}
Assertions (a) and (b) are well known and easy to prove.
Assertion (c) follows from (a), (b) and \ref{o83020192djwoidja}.
\end{proof}

\begin{parag} \label {7832t8763trhxg0nfhn}
(Refer to \cite{Lang52QuasiAlgClos} for this paragraph.)
A field $K$ is said to be $C_1$ if, for every choice of integers $0<d<n$ and
every homogeneous polynomial $F(X_1, \dots, X_n) \in K[X_1, \dots, X_n]$ of degree $d$,
there exists $(a_1, \dots, a_n) \in K^n \setminus \{(0,\dots,0)\}$ satisfying $F(a_1,\dots,a_n)=0$.
Tsen's Theorem states that if $K$ is a function field in one variable over an algebraically closed field,
then $K$ is $C_1$.
Lang showed that if a field $K$ is $C_1$ then so is every algebraic extension of $K$.
It follows in particular:
\begin{quote}
\it
If $\bk$ is an algebraically closed field and $\tau$ an indeterminate over $\bk$ then
$\bk(\tau)$ is a $C_1$ field. Moreover, if $\Char\bk=p>0$ then $\bk(\tau)^{p^{-\infty}}$ is $C_1$.
\end{quote}
\end{parag}

\begin{parag} \label {823o87cry1i234eyhfdqwiue}
{\it Let $L/K$ be a function field in one variable, where $K$ is a $C_1$ field and is algebraically closed in $L$.
Then $L/K$ is rational if and only if it has genus $0$.}

Indeed, it is known that if $L/K$ has genus zero then it is the function field of a curve in $\proj^2_K$ given by
an equation $F(X,Y,Z)=0$, where $F(X,Y,Z) \in K[X,Y,Z]$ is an irreducible homogeneous polynomial of degree $2$.
As $K$ is $C_1$, the curve has a $K$-rational point, so $L/K$ has a place of degree $1$ and hence is rational.
The converse is clear.
\end{parag}

\medskip
\begin{parag} \label {9087f8912hwedjnawoe}
Let $\bk$ be an algebraically closed field and $f : X \to Y$ a dominant morphism of integral schemes of
finite type over $\bk$.  Assume that $\dim X = \dim Y$.
Then $\bk(X)/\bk(Y)$ is a finite extension of fields, where 
$\bk(X)$ and $\bk(Y)$ denote the function fields of $X$ and $Y$ respectively.
One defines
$$
\text{$\deg f = [ \bk(X) : \bk(Y)]$,\ \  $\deg_s f = [ \bk(X) : \bk(Y)]_s$\ \  and\ \  $\deg_i f = [ \bk(X) : \bk(Y)]_i$.}
$$
It is well known (cf.\ \cite[Prop.\ 9.7.8, p.\ 82]{EGA4.3} and \cite[D\'ef.\ 4.5.2, p.\ 61]{EGA4.2})
that the positive integer $d = \deg_s f$ has the following property:
\begin{quote}
\it
There exists a nonempty open subset $V \subseteq Y$ such that,
for each closed point $y \in V$, the set $f^{-1}(y)$ consists of exactly $d$ closed points of $X$.
\end{quote}
\end{parag}

The following notation is used in \ref{1po239092up938rnhfjsdu48}.
Given morphisms of schemes $X \xrightarrow{f} Y \xrightarrow{\pi} T$ and a point $P \in T$,
we write $X_P = X \times_T \Spec\kappa(P)$ and $Y_P = Y \times_T \Spec\kappa(P)$
for the fibers of $\pi \circ f$ and $\pi$ over $P$ (where $\kappa(P)$ is the residue field of $T$ at $P$).
Note the commutative diagram
$$
\xymatrix@R=13pt{
X_P \ar[r]^{f_P} \ar[d] &  Y_P \ar[r]^-{\pi_P} \ar[d] &  \Spec\kappa(P) \ar[d]  \\
X \ar[r]_{f} &  Y \ar[r]_{\pi} &  T
}
$$
in which every square is a pullback square.

\begin{lemma}  \label {1po239092up938rnhfjsdu48}
Let $\bk$ be an algebraically closed field and $X \xrightarrow{f} Y \xrightarrow{\pi} T$  dominant 
morphisms of integral schemes of finite type over $\bk$.
Suppose that $\dim X = \dim Y$ and that
\begin{equation*}
\begin{minipage}[t]{.9\textwidth}
there exists a nonempty open subset $U \subseteq T$ such that,
for each closed point $P \in U$, $X_P$ and $Y_P$ are integral schemes. 
\end{minipage}
\end{equation*}
Then there exists a nonempty open set $U' \subseteq U$ such that, for every closed point $P \in U'$,
$$
\text{$f_P : X_P \to Y_P$ is dominant, $\dim X_P = \dim Y_P$ and $\deg_s( f_P ) = \deg_s(f)$.}
$$
\end{lemma}

\begin{proof}
In lack of a suitable reference, we provide a proof.
For each closed point $P \in U$, $X_P$ and $Y_P$ are closed subschemes of $X$ and $Y$ respectively.
Viewing them as subsets of $X$ and $Y$, we have
$Y_P = \pi^{-1}(P)$, $X_P = (\pi \circ f)^{-1}(P) = f^{-1}( Y_P )$
and the continuous map $f_P : X_P \to Y_P$ is simply the restriction of $f$.
Note that $f_P^{-1}(y) = f^{-1}(y)$ for all $y \in Y_P$.

Let $d = \deg_s(f)$ and choose a nonempty open set $V \subseteq \pi^{-1}(U)$ such that,
for each closed point $y \in V$, the set $f^{-1}(y)$ consists of exactly $d$ closed points of $X$ 
(cf.\ \ref{9087f8912hwedjnawoe}).
Then $\pi(V)$ is dense in $T$ and hence contains a nonempty open subset $U'$ of $T$.
Note that $U' \subseteq U$.

Let $P$ be a closed point of $U'$.  Then $Y_P \cap V \neq \emptyset$ (because $U'\subseteq \pi(V)$) and,
for every closed point $y \in Y_P \cap V$,
the set $f_P^{-1}(y)$ consists of exactly $d$ closed points of $X_P$.
Since $f_P : X_P \to Y_P$ is a morphism of integral schemes of finite type over $\bk$,
it follows that $f_P$ is dominant, that $\dim X_P = \dim Y_P$ and (by \ref{9087f8912hwedjnawoe} again)
that $\deg_s( f_P ) = d$, as desired.
\end{proof}

The following result is proved in paragraphs 2.8--3.3 of \cite{Rus:fg}.

\begin{lemma}  \label {x92938up92fjeidfq}
Let $\bk$ be an algebraically closed field, $A=\kk2$ and $F \in A \setminus \bk$.
Assume that $\bk(F)$ is algebraically closed in $\Frac A$ and let $g$ denote the genus
of the function field $\Frac A\,/\, \bk(F)$.
Then, for any diagram \eqref{7y74r47r77392392093jeij} as in definition~\ref{ij920201e9d2jxd}, the following holds:
\begin{equation*}  
\begin{minipage}[t]{.9\textwidth}
For almost all closed points $P \in \proj^1$, 
the arithmetic genus of the curve $\bar f^{-1}(P)$ is equal to~$g$.
\end{minipage}
\end{equation*}
\end{lemma}

\begin{proof}
Since this fact is not explicit in \cite{Rus:fg}, we fill the gaps.
Choose a diagram \eqref{7y74r47r77392392093jeij}.
The assumption that $\bk(F)$ is algebraically closed in $\Frac A$ 
implies that,
for almost all closed points $P \in \proj^1$, $\bar f^{-1}(P)$ is an integral curve over $\bk$.
Note that the number
``arithmetic genus of $\bar f^{-1}(P)$ for a general closed point $P \in \proj^1$''
is independent of the choice of a diagram \eqref{7y74r47r77392392093jeij}
(any two diagrams can be reconciled after finitely many extra blowings-up, and these
blowings-up affect only finitely many fibers $\bar f^{-1}(P)$).
So it's enough to show that at least one diagram \eqref{7y74r47r77392392093jeij} has the desired property.

Choose $x,y$ such that $A=\bk[x,y]$, let $d = \deg(F)$ (with respect to $x,y$),
let $F^* \in \bk[x,y,z]$ be the homogenization of $F$, and consider the pencil
$\Lambda(F) =  \setspec{ \div_0( aF^* + bZ^d ) }{ (a:b) \in \proj^1 }$ on $\proj^2$
(where we write $\div_0(H)$ for the divisor of zeroes of a homogeneous polynomial $H \in \bk[X,Y,Z] \setminus \{0\}$).
The assumption that $\bk(F)$ is algebraically closed in $\Frac A$ implies that the general member of $\Lambda(F)$
is irreducible and reduced.
Let $B$ be the set of base points of $\Lambda(F)$, including infinitely near ones.
Then $B$ is a finite set.
Let $\pi : X \to \proj^2$ be the blowing-up of $\proj^2$ along $B$ (i.e., resolve the base points of $\Lambda(F)$);
then $X$ is a nonsingular projective surface,
$\pi$ is a birational morphism centered at points of $\proj^2 \setminus \aff^2$
and the strict transform of $\Lambda(F)$ on $X$ is free of base points.
This base point free pencil determines a morphism $\bar f : X \to \proj^1$;
by restricting $\pi$ we get an isomorphism $\pi^{-1}(\aff^2) \to \aff^2$,
whose inverse defines an open immersion $\aff^2 \hookrightarrow X$;
so we have constructed a diagram \eqref{7y74r47r77392392093jeij}.
By paragraphs 2.8--3.3 of \cite{Rus:fg}, 
the genus $g$ of the function field $\Frac A\,/\, \bk(F)$ is equal to 
\begin{equation} \label {ff92d039uejeori102}
(d-1)(d-2)/2 - \sum_{Q\in B} \mu(Q)(\mu(Q)-1)/2,
\end{equation}
where $\mu(Q)$ is the multiplicity of the base point $Q$, i.e., the multiplicity of $Q$ on the general member of
a suitable strict transform of $\Lambda(F)$ (refer to \cite{Rus:fg} for details). 
Clearly, the number \eqref{ff92d039uejeori102} is equal to 
the arithmetic genus of $\bar f^{-1}(P)$ for a general closed point $P \in \proj^1$.
\end{proof}

\section{Proofs}
\label {Sec:Proofs}

Throughout this section, $\bk$ is an algebraically closed field and $F \in A = \kk2$. 
We also consider the $\bk(F)$-algebra $\Aeul = S^{-1}A$ where $S = \bk[F] \setminus \{0\}$.
Define $q=1$ if $\Char\bk=0$, and $q=p$ if $\Char\bk=p>0$.

Given $\bk$-domains $B \subseteq C$ and $x \in C$, the phrase ``$x$ is purely inseparable over $B$''
means that there exists $n\in\Nat$ such that $x^{q^n} \in B$ (if $\Char\bk=0$, this means that $x \in B$).
We also define ``$C$ is purely inseparable over $B$'' to mean that each element of $C$ is purely inseparable over $B$.
When $B$ and $C$ are fields, these definitions coincide with the usual ones.

Let us also remark that if $F/E$ is a purely inseparable extension of fields, 
$\Oeul$ a valuation ring of $F$ and $\Oeul' = \Oeul \cap E$,
then $\Oeul$ is purely inseparable over $\Oeul'$  and consequently
the residue field of $\Oeul$ is a purely inseparable extension of that of $\Oeul'$;
moreover, every valuation ring of $E$ has a {\bf unique} extension to a valuation ring of $F$.

\medskip
\begin{proof}[Proof that \ref{o9823ry91c83yrquwei}\eqref{989kjbgcdsaw4567u9jdf} implies
\ref{o9823ry91c83yrquwei}\eqref{982373253hhkj947}.]
Suppose that \ref{o9823ry91c83yrquwei}\eqref{989kjbgcdsaw4567u9jdf} holds.
Choose $G \in \Frac A$ such that $\Frac A$ is purely inseparable over $\bk(F,G)$.

Consider any $W \in A$ such that $\bk[F] \subseteq \bk[W] \subset A$.
Then $W$ is integral over $\bk[F]$ and  $W^{q^n} \in \bk(F,G)$ for some $n$.
Since $W^{q^n} \in \bk(F,G)$ and $W^{q^n}$ is integral over $\bk[F]$, we have 
$W^{q^n} \in \bk[F]$ and hence  $\bk[W^{q^n}] \subseteq \bk[F] \subseteq \bk[W]$;
by assumption, $F \notin A^p$ if $\Char\bk=p>0$; so $\bk[F]=\bk[W]$.
Consequently, \ref{Matsudp23rwije} implies:
\begin{center}
$F-\lambda$ is irreducible in $A$ for almost all $\lambda \in \bk$.
\end{center}

Choose $H \in A \setminus \{0\}$ such that $\bk[F,G] \subseteq A_H$.

For almost all $\lambda \in \bk$, 
$F-\lambda$ is irreducible in $A$ and $F-\lambda \nmid H$ in $A$;
for each such $\lambda$, $F-\lambda$ is irreducible in $A_H$.
Consequently, the morphisms of schemes $\Spec A_H \to \Spec \bk[F,G] \to \Spec \bk[F]$
(determined by the inclusions $A_H \supseteq \bk[F,G] \supset \bk[F]$)
satisfy the hypothesis of \ref{1po239092up938rnhfjsdu48}.
This implies that there exists a subset $U$ of $\bk$ such that $\bk \setminus U$ is a finite set and,
for all $\lambda \in U$,
\begin{gather}
\notag \text{$\bk[F,G]/(F-\lambda)\bk[F,G] \to A_H/(F-\lambda)A_H$ is injective and} \\
\label {290eh07c481yp9284dj} 
\big[ L_\lambda : K_\lambda \big]_s = \big[ \Frac A_H : \bk(F,G) \big]_s = \big[ \Frac A : \bk(F,G) \big]_s = 1
\end{gather}
where we set $L_\lambda = \Frac\big( A_H/(F-\lambda)A_H \big)$
and  $K_\lambda = \Frac\big( \bk[F,G]/(F-\lambda)\bk[F,G] \big)$.
For each $\lambda \in U$ we have  $\bk \subset K_\lambda \subseteq L_\lambda$ where each of $K_\lambda$, $L_\lambda$ 
is a function field in one variable over the algebraically closed field $\bk$ and,
by \eqref{290eh07c481yp9284dj}, $L_\lambda / K_\lambda$ is purely inseparable;
thus  \cite[Ch.\ IV, 2.5]{Hartshorne} implies that $L_\lambda/\bk$ and $K_\lambda/\bk$ have the same genus,
which is $0$ since $K_\lambda = \bk^{(1)}$. Hence, 
$$
L_\lambda = \bk^{(1)} \quad \text{for almost all $\lambda \in \bk$.}
$$
Now $L_\lambda = \Frac\big( A_H/(F-\lambda)A_H \big) = \Frac\big( A/(F-\lambda)A \big)$,
so $A/(F-\lambda)A$ is $\bk$-rational for almost all $\lambda \in \bk$, i.e.,
we have shown that \ref{o9823ry91c83yrquwei}\eqref{989kjbgcdsaw4567u9jdf} implies
\ref{o9823ry91c83yrquwei}\eqref{982373253hhkj947}.
\end{proof}

\medskip
\begin{proof}[Proof that \ref{o9823ry91c83yrquwei}\eqref{982373253hhkj947} implies
\ref{o9823ry91c83yrquwei}\eqref{989kjbgcdsaw4567u9jdf}.]
Let $F$ be a generally rational polynomial in $A$.
The assumption implies, in particular, that there exists $\lambda \in \bk$ such that $F-\lambda$ is irreducible
in $A$; so if $\Char\bk=p>0$ then $F \notin A^p$ (which is part of the desired conclusion).
Let $\tau$ be an indeterminate over $\bk$, let $\overline K$ be an algebraic closure of $\bk(\tau)$ and let
$$
K  = \setspec{ x \in \barr K }{ \text{$x$ is purely inseparable over $\bk(\tau)$} } = \begin{cases}
\bk(\tau), & \text{if $\Char\bk=0$}, \\
\bk(\tau)^{p^{-\infty}}, & \text{if $\Char\bk=p>0$.}
\end{cases}
$$
Then $\bk(\tau) \subseteq K \subset \barr K$ and 
$$
F \in A =\bk[X,Y] \subset \bk(\tau)[X,Y] \subseteq K[X,Y] \subset \barr K[X,Y].
$$

Applying \ref{f9032978y9187x4yriwedj} to 
$F \in \bk[X,Y] \subset \barr K[X,Y]$ shows that $F$ is a generally rational polynomial in $\barr K[X,Y]$.
Then almost all $\lambda \in \barr K$ are such that
$\barr K[X,Y]/(F-\lambda)$ is $\barr K$-rational.
Since $\setspec{\tau +\lambda}{ \lambda \in \bk }$ is an infinite subset of $\barr K$,
there exists $\lambda \in \bk$ such that
$\barr K[X,Y]/(F-\tau-\lambda)$ is $\barr K$-rational.
There exists a $\bk$-automorphism $\theta$ of $\barr K$ which sends $\tau +\lambda$ on $\tau $.
Extend $\theta$ to a $\bk$-automorphism $\Theta$ of $\barr K[X,Y]$ such that
$\Theta(X)=X$ and $\Theta(Y)=Y$. Then $\Theta : \barr K[X,Y] \to \barr K[X,Y]$ is an
isomorphism of rings satisfying $\Theta( \barr K) = \barr K$ and $\Theta(F-\tau -\lambda) = F-\tau $;
it induces an isomorphism of rings 
$$
\barr K[X,Y]/(F-\tau -\lambda) \to \barr K[X,Y]/(F-\tau)
$$
which maps $\barr K$ onto itself.
As $\barr K[X,Y]/(F-\tau -\lambda)$ is $\barr K$-rational,
\begin{equation} \label {pf192p39r23jwdji}
\text{$\barr K[X,Y]/(F-\tau)$ is $\barr K$-rational}.
\end{equation}
This implies, in particular, that $F-\tau$ is irreducible in $\barr K[X,Y]$; then
it is also irreducible in $K[X,Y]$ and in $\bk(\tau)[X,Y]$.
Moreover, 
\begin{align*}
(F-\tau)\barr K[X,Y] \cap K[X,Y] & = (F-\tau) K[X,Y], \\ 
(F-\tau)K[X,Y] \cap \bk(\tau)[X,Y] & = (F-\tau) \bk(\tau)[X,Y] 
\end{align*}
because, say,  $\bk(\tau)[X,Y] \to K[X,Y] \to \barr K[X,Y]$ are faithfully flat homomorphisms
(if $R \to S$ is a faithfully flat homomorphism and $I$ is an ideal of $R$ then $IS \cap R = I$).
So there is a commutative diagram of integral domains and injective homomorphisms
\begin{equation}  \label {9f923p0912p093r}
\xymatrix{
\barr K \ar[r]   
&   \barr R  \ar[r]  
&   M   
& \barr R = \barr K[X,Y]/(F-\tau), \quad M = \Frac \barr R   \\
K  \ar[r] \ar[u]  &   R  \ar[r] \ar[u]  &   L \ar[u] 
& R = K[X,Y]/(F-\tau), \quad L = \Frac R \\
\bk(\tau)  \ar[r] \ar[u]  & R_0  \ar[r] \ar[u]  &   L_0 \ar[u]
& R_0 = \bk(\tau)[X,Y]/(F-\tau)  , \quad L_0 = \Frac R_0 .
}
\end{equation}
Applying the exact functor $\barr K \otimes_K (\underline{\ \,})$ to 
$0 \to (F-\tau) \to K[X,Y] \to R \to 0$ yields
$0 \to (F-\tau) \to \barr K[X,Y] \to \barr R \to 0$, and this shows that
$\barr R = \barr K \otimes_K R$.
Note that $L=\Sigma^{-1}R$ where  $\Sigma = R\setminus\{0\}$.
Since $\barr K$ is integral over $K$ and $\barr R = \barr K \otimes_K R$, $\barr R$  is integral over $R$
and consequently $\Sigma^{-1}\barr R$ is integral over $\Sigma^{-1}R$ ($=L$); so $\Sigma^{-1}\barr R$ is a field, i.e.,
$\Sigma^{-1}\barr R = M$. So we have shown that 
$\barr R = \barr K \otimes_K R$ and $M = \barr R \otimes_R L$.
The same argument shows that 
$R = K \otimes_{\bk(\tau)} R_0$ and $L = R \otimes_{R_0} L_0$.
This can be summarized by saying that the four little squares, in diagram~\eqref{9f923p0912p093r}, are
pushout squares; so
\begin{equation}  \label {fp2ip3r91998xp2n8rmhe}
\text{all nine  squares, in \eqref{9f923p0912p093r}, are pushout squares.}
\end{equation}
The following fact is well known: {\it suppose that $B,F,X,Y$ are rings,
\begin{equation*}
\xymatrix{
F \ar[r] &  Y  \\
B  \ar[r] \ar[u]  &  X \ar[u]  
}
\end{equation*}
is a pushout square (i.e., $X \otimes_B F = Y$) in which all arrows are injective homomorphisms of rings,
$F$ is a free $B$-module and there exists a basis $\Beul$ of $F$ over $B$ such that $1 \in \Beul$;
then $Y$ is a free $X$-module, there exists a basis $\Beul'$ of $Y$ over $X$ such that $1 \in \Beul'$
and $F \cap X=B$, when we view $B,F,X$ as subsets of $Y$.} Applying this to 
\eqref{9f923p0912p093r} and \eqref{fp2ip3r91998xp2n8rmhe} gives, in particular:
\begin{equation} \label {pf92p39rupidi}
\barr K \cap L = K,
\quad \barr K \cap L_0 = \bk(\tau),
\quad K \cap L_0 = \bk(\tau),
\quad \barr R \cap L = R \text{\ \ and\ \ } R \cap L_0 = R_0.
\end{equation}
In view of the fact that $\barr K$ is algebraically closed in $M$, 
the equalities $\barr K \cap L = K$ and $\barr K \cap L_0 = \bk(\tau)$ imply:
\begin{equation} \label {5923ueip23fnijq3}
\text{$K$ is algebraically closed in $L$ and $\bk(\tau)$ is algebraically closed in $L_0$}.
\end{equation}

Note that $K$ is purely inseparable over $\bk(\tau)$; since $L = K \otimes_{\bk(\tau)} L_0$, $L$ is the compositum
$KL_0$ and it follows that $L$ is purely inseparable over $L_0$; since $R \cap L_0 = R_0$, we obtain that 
$R$ is purely inseparable over $R_0$.
We record this:
\begin{equation} \label {2p93pc892h3pfued}
\text{$L$ (resp.\ $R$) is purely inseparable over $L_0$ (resp.\ $R_0$)}.
\end{equation}

Observe in particular that the following assertions are true:
\begin{enumerate}

\item[(i)] $M/\barr K$ is a function field in one variable and $\barr K$ is algebraically closed in $M$;

\item[(ii)] $L/K$ is a function field in one variable and $K$ is algebraically closed in $L$;

\item[(iii)] the compositum of fields $\barr K L$ is equal to $M$;

\item[(iv)] $M$ is an algebraic extension of $L$;

\item[(v)] $K$ is perfect.

\end{enumerate}
By \cite[Thm~III.6.3]{StichtenothBook}, conditions (i--v) imply that $M/\barr K$ has the same genus as $L/K$;
as $M = \barr K^{(1)}$ by \eqref{pf192p39r23jwdji}, that genus is $0$.
Now $K$ is a $C_1$ field by \ref{7832t8763trhxg0nfhn};
so \ref{823o87cry1i234eyhfdqwiue} yields:
\begin{equation} \label {9f92039c32cruf}
L=K^{(1)}.
\end{equation}
Choose $v \in L$ such that $L=K(v)$.
If $\Char\bk=0$, define $g=v$; 
if $\Char\bk=p>0$, define  $g = v^{p^n}$ where $n \in \Nat$ is large enough to have $v^{p^n} \in L_0$.
Then in both cases we have $g \in L_0$, and we claim:
\begin{equation} \label {923r9982r1hd18rnhdnh}
\text{$L_0$ is a purely inseparable extension of $\bk(\tau,g)$.}
\end{equation}

Indeed, if $\Char\bk=0$ then $L_0=L=K(v)=\bk(\tau,v)=\bk(\tau,g)$, so \eqref{923r9982r1hd18rnhdnh} holds.
Assume that $\Char\bk=p>0$.
We use the following notation.  Given $s \in \Nat$ and a polynomial 
$P(T) = \sum_i a_i T^i \in K[T] = K^{[1]}$ (where $a_i \in K$), let $P^{(p^s)}(T) = \sum_i a_i^{p^s} T^i \in K[T]$.
Note that $P^{(p^s)}(T) \in \bk(\tau)[T]$ if $s$ is large enough.

Let $\xi \in L_0$. 
Then $\xi \in L = K(v)$, so $\xi = P(v)/Q(v)$ for some $P(T), Q(T) \in K[T]$, $Q(T) \neq 0$.
Choose $s \ge n$ large enough to have $P^{(p^s)}(T), Q^{(p^s)}(T) \in \bk(\tau)[T]$.
Then 
$$
\xi^{p^s}
=  P^{(p^s)}(v^{p^s}) / Q^{(p^s)}(v^{p^s}) 
=  P^{(p^s)}(g^{p^{s-n}}) / Q^{(p^s)}(g^{p^{s-n}}) \in \bk(\tau,g),
$$
showing that $\xi$ is purely inseparable over $\bk(\tau,g)$. This proves \eqref{923r9982r1hd18rnhdnh}.

Finally, we note that there is a commutative diagram:
\begin{equation} \label {9823r723txhdi2ed}
\xymatrix{
\bk(F) \ar @{^{(}->} [r]  &  \Aeul \ar @{^{(}->} [r]   &  \bk(X,Y) \ar @{=}[r] & \Frac A  \\
\bk(\tau) \ar @{^{(}->} [r] \ar[u]^-{\isom}   &  R_0 \ar @{^{(}->} [r] \ar[u]^-{\isom}  &  L_0 \ar[u]^-{\isom}_-{\phi} 
}
\end{equation}
where the vertical arrows are $\bk$-isomorphisms that send $\tau$ to $F$
and where $\Aeul = S^{-1}A$, $S = \bk[F] \setminus \{0\}$.
Let $g \in L_0$ be as before and
let $G = \phi(g) \in \bk(X,Y)$.
Then \eqref{923r9982r1hd18rnhdnh} implies that $\bk(X,Y)$ is purely inseparable over $\bk(F,G)$.

This shows that \ref{o9823ry91c83yrquwei}\eqref{982373253hhkj947} implies
\ref{o9823ry91c83yrquwei}\eqref{989kjbgcdsaw4567u9jdf}
and completes the proof of \ref{o9823ry91c83yrquwei}.
\end{proof}

All facts established in the proof of
\ref{o9823ry91c83yrquwei}\eqref{982373253hhkj947} $\Rightarrow$ \ref{o9823ry91c83yrquwei}\eqref{989kjbgcdsaw4567u9jdf}
are valid whenever $F$ is a generally rational polynomial in $A$.
This is used in several of the proofs below.

\medskip
\begin{proof}[Proof of \ref{q9830732847ryt23rhf}]
If \eqref{k3kk34k2j3ki3jo3i4uy7} or \eqref{8788sx88x8x8x878x7x8x7f} holds then $F$ is a generally rational
polynomial in $A$ (this is obvious if \eqref{k3kk34k2j3ki3jo3i4uy7} holds and
is a consequence of \ref{o9823ry91c83yrquwei} if \eqref{8788sx88x8x8x878x7x8x7f} holds,
since a field generator in $A$ cannot belong to $A^p$ if $\Char\bk=p>0$).
So, to prove the theorem, we may assume throughout that $F$ is a generally rational polynomial in $A$.

Let $g$ denote the genus of the function field $\Frac A\,/\, \bk(F)$ and note that $\bk(F)$ is
algebraically closed in $\Frac A$ (for instance by \eqref{5923ueip23fnijq3} and \eqref{9823r723txhdi2ed},
which are valid here since  $F$ is a generally rational polynomial in $A$).
Now $F$ is a field generator if and only if $\Frac A = \bk(F)^{(1)}$, and this is equivalent to $g=0$ by 
\ref{7832t8763trhxg0nfhn} and \ref{823o87cry1i234eyhfdqwiue}.
So it's enough to show:
\begin{equation} \label {if23i23id9ijqdmn}
\text{$g=0$ if and only if $(F,A)$ does not have moving singularities.}
\end{equation}

Choose a diagram \eqref{7y74r47r77392392093jeij} as in definition~\ref{ij920201e9d2jxd}.
Then, for almost all closed points $P \in \proj^1$, $\bar f^{-1}(P)$ is an integral curve over $\bk$.
By \ref{x92938up92fjeidfq}, 
\begin{equation*}  \label {nvb9f20930192udoiwjda}
\begin{minipage}[t]{.9\textwidth}
for almost all closed points $P \in \proj^1$, 
the arithmetic genus of $\bar f^{-1}(P)$ is equal to $g$.
\end{minipage}
\end{equation*}
So, keeping in mind that $\bar f^{-1}(P)$ is rational, we see that
$g=0$
iff
the arithmetic genus of $\bar f^{-1}(P)$ is equal to $0$ for almost all closed points $P \in \proj^1$,
iff
$\bar f^{-1}(P)$ is nonsingular for almost all closed points $P \in \proj^1$,
iff
$(F,A)$ does not have moving singularities, proving \eqref{if23i23id9ijqdmn}.
\end{proof}

\medskip
\begin{proof}[Proof of \ref{9230r9984tyqrawdc9r0fhd}]
Let $F$ be a pseudo field generator in $A$.
If $\Char\bk=p>0$ then $F$ is good if and only if $F^p$ is good, and it is easy to check that $F$ and $F^p$ have
exactly the same set of dicriticals and that a given dicritical is a p.i.\ dicritical of $F$ iff it is a
p.i.\ dicritical of $F^p$;
so, to prove \ref{9230r9984tyqrawdc9r0fhd} in characteristic $p>0$,
we may (and shall) assume that $F \notin A^p$. 
Then, by \ref{o9823ry91c83yrquwei}, $F$ is a generally rational polynomial in $A$.
Consequently, all facts established in the proof of
\ref{o9823ry91c83yrquwei}\eqref{982373253hhkj947} $\Rightarrow$ \ref{o9823ry91c83yrquwei}\eqref{989kjbgcdsaw4567u9jdf}
remain valid here.

Suppose that $F$ is good.
Then there exists $G \in A$ such that $\Frac A$ is purely inseparable over $\bk(F,G)$.
Let $(\Reul,\mgoth)$ be the unique valuation ring of $\bk(F,G)/\bk(F)$ such that $G \notin \Reul$
and note that $\Reul/\mgoth = \bk(F)$.
Since $\Frac A$ is purely inseparable over $\bk(F,G)$, it follows that $(\Reul,\mgoth)$ extends uniquely 
to a valuation ring $(\Seul,\ngoth)$ of $\Frac A/ \bk(F)$ and that $\Seul/\ngoth$ is a purely inseparable
extension of  $\Reul/\mgoth = \bk(F)$. Then $\Seul \in \proj_\infty( \Aeul/\bk(F) )$ is a purely inseparable
dicritical of $F$ (where $\Aeul = S^{-1}A$, $S = \bk[F] \setminus \{0\}$, as before),
proving that  \ref{9230r9984tyqrawdc9r0fhd}\eqref{928fhnc1bci3urh}
implies \ref{9230r9984tyqrawdc9r0fhd}\eqref{432kln9c745jdie}.

For the converse, begin by observing that
the isomorphism $\phi : L_0 \to \Frac A$ of \eqref{9823r723txhdi2ed} satisfies
$\phi^{-1}(\Aeul)=R_0$ and $\phi^{-1}(\bk(F))=\bk(\tau)$.
Suppose that $F$ has at least one purely inseparable dicritical  $\Seul \in \proj_\infty( \Aeul/\bk(F) )$.
Then $\phi^{-1}( \Seul )$ is an element of $\proj_\infty( R_0 / \bk(\tau) )$  which we denote
$(\Oeul_0, \mgoth_0)$; $\Seul$ being a purely inseparable dicritical, the residue field of $\Seul$ is 
purely inseparable over $\bk(F)$ and consequently
$\Oeul_0 / \mgoth_0$ is a purely inseparable extension of $\bk(\tau)$.
As (by \eqref{2p93pc892h3pfued}) $L$ is purely inseparable over $L_0$, 
$(\Oeul_0, \mgoth_0)$ extends uniquely to a valuation ring $(\Oeul, \mgoth)$  of $L$ over $K$
and  $\Oeul/\mgoth$  is a purely inseparable extension of $\Oeul_0/\mgoth_0$:
$$
\xymatrix{
K \ar[r]  &  \Oeul / \mgoth  \\
\bk(\tau)  \ar[r]_-{\text{p.i.}} \ar[u]^{\text{p.i.}}  &  \Oeul_0 / \mgoth_0 \ar[u]_{\text{p.i.}} 
}
$$
Then $\Oeul / \mgoth$ is purely inseparable over $K$.
Since $K$ is perfect, $\Oeul / \mgoth = K$.
Since $L=K^{(1)}$ by \eqref{9f92039c32cruf}, it follows that the ring
$$
\Reul = \bigcap \big( \proj(L/K) \setminus \{ \Oeul \} \big)
$$
satisfies $\Reul = K^{[1]}$ and $L=\Frac \Reul$.
Choose $v$ such that $\Reul = K[v]$.
Then $L=K(v)$, so if we define $g \in L_0$ as in the proof of \ref{o9823ry91c83yrquwei}
(see just before \eqref{923r9982r1hd18rnhdnh}),
and if we take $G = \phi(g) \in \Frac A$, then the proof of 
\ref{o9823ry91c83yrquwei}\eqref{982373253hhkj947} $\Rightarrow$ \ref{o9823ry91c83yrquwei}\eqref{989kjbgcdsaw4567u9jdf}
shows that $\Frac A$ is purely inseparable over $\bk(F,G)$.
Note that $g = v^{q^n}$ for some $n \in \Nat$, so $g \in \Reul$.

Since $L_0 \subseteq L$ and $\bk(\tau) \subseteq K$, we have a well defined map
\begin{equation}  \label {pf9u2p93pc1nudqwdas}
\proj(L/K) \to \proj(L_0/\bk(\tau)), \quad B \mapsto B \cap L_0;
\end{equation}
this map is surjective because $K/\bk(\tau)$ is an algebraic extension;
it is injective because $L/L_0$ is purely inseparable; so \eqref{pf9u2p93pc1nudqwdas} is bijective.
It follows that the image of $\proj(L/K) \setminus \{ \Oeul \}$ by that map is equal to
$\proj(L_0/\bk(\tau)) \setminus \{ \Oeul_0 \}$, and this implies that
\begin{equation} \label {87h8s9xc8v6ruh}
\Reul \cap L_0 = \bigcap \big( \proj(L_0/\bk(\tau)) \setminus \{ \Oeul_0 \} \big).
\end{equation}
As $\proj(L_0/\bk(\tau)) \setminus \{ \Oeul_0 \} \supseteq \proj_{\text{\rm fin}}(R_0/\bk(\tau))$, we get
\begin{equation} \label {7y4hn30kfmqjcj}
\bigcap \big( \proj(L_0/\bk(\tau)) \setminus \{ \Oeul_0 \} \big)
\subseteq \bigcap \proj_{\text{\rm fin}}(R_0/\bk(\tau)) = \overline R_0,
\end{equation}
where  $\overline R_0$ is the integral closure of $R_0$ in $L_0$.
In view of diagram~\eqref{9823r723txhdi2ed} and of the fact that $\Aeul$ is integrally closed in $\Frac A$,
we see that $R_0$ is a normal domain,
so $\overline R_0=R_0$ and hence (by \eqref{87h8s9xc8v6ruh} and \eqref{7y4hn30kfmqjcj}) $\Reul \cap L_0 \subseteq R_0$.
As $g \in \Reul \cap L_0$, we have $G = \phi(g) \in \phi(R_0) = \Aeul = S^{-1}A$.
Multiplying $G$ by a suitable element of $S = \bk[F]\setminus\{0\}$ gives an element $G' \in A$,
and since $\bk(F,G') = \bk(F,G)$, $\Frac A$ is purely inseparable over $\bk(F,G')$.
So $F$ is good, and this completes the proof of \ref{9230r9984tyqrawdc9r0fhd}.
\end{proof}

Before proving  \ref{020399f09fjcnEH1283}, we need a definition and a lemma.
See \ref{q9nc29991209fh27gdaj} for the notation.

\begin{definition} \label {9239f9cn8d3ndha}
We say that a function field in one variable $F/E$ {\it has property $(*)$} if:
\begin{equation} \tag{$*$}
\begin{minipage}[t]{.9\textwidth}
For any choice of distinct elements $\Oeul_1, \Oeul_2 \in \proj(F/E)$,
there exists $\xi \in F \setminus E$ such that $\supp(\div \xi) = \{ \Oeul_1, \Oeul_2 \}$.
\end{minipage}
\end{equation}
\end{definition}

We leave it to the reader to check that if $F=E^{(1)}$ then $F/E$ has property $(*)$.

\begin{lemma}  \label {pf9u293jdnaKSFJ}
Let $\bk$ be an algebraically closed field and
let $F \in A = \kk2$ be a generally rational polynomial in $A$.
Then the function field $\Frac(A)/\bk(F)$ has property $(*)$.
\end{lemma}

\begin{proof}
Let $F$ be a generally rational polynomial of $A$.
Then the facts established in the proof
that \ref{o9823ry91c83yrquwei}\eqref{982373253hhkj947} implies
\ref{o9823ry91c83yrquwei}\eqref{989kjbgcdsaw4567u9jdf} are valid here.
The notation being as in that proof,
consider the two function fields in one variable $L_0/\bk(\tau)$ and  $L/K$.
Since $L = K^{(1)}$ by \eqref{9f92039c32cruf}, $L/K$ has property $(*)$.
We noted in \eqref{2p93pc892h3pfued} and \eqref{pf9u2p93pc1nudqwdas} that $L/L_0$ is purely inseparable
and that the map $\proj(L/K) \to \proj(L_0/\bk(\tau))$, $\Oeul \mapsto \Oeul \cap L_0$, is bijective.
It easily follows that $L_0/\bk(\tau)$ has property $(*)$.
In view of the isomorphisms of \eqref{9823r723txhdi2ed}, we conclude that the
function field $\Frac(A)/\bk(F)$ has property $(*)$.
\end{proof}

\medskip
\begin{proof}[Proof of \ref{020399f09fjcnEH1283}.]
Let $F$ be a generally rational polynomial of $A$.
Once more, all facts established in the proof of
\ref{o9823ry91c83yrquwei}\eqref{982373253hhkj947} $\Rightarrow$ \ref{o9823ry91c83yrquwei}\eqref{989kjbgcdsaw4567u9jdf}
remain valid here.
Let $W=\Frac A$ and $\Aeul = S^{-1}A$ where $S = \bk[F] \setminus \{0\}$.

Consider the finite set $\Lambda = \setspec{ \lambda \in \bk }{\text{$F-\lambda$ is not irreducible in $A$}}$.
For each $\lambda \in \Lambda$, choose a prime factorization of $F-\lambda$ in $A$,
$F-\lambda = \prod_{j=1}^{n_\lambda} G_{\lambda,j}^{e_{\lambda,j}}$,
where the $G_{\lambda,j}$ are pairwise relatively prime irreducible elements of $A$,
and where $e_{\lambda,j}>0$ for all $\lambda,j$.
Note that $n_\lambda$ has the same meaning here as in the statement of the theorem.
Let $\Geul_\lambda = \{ G_{\lambda,1}, \dots, G_{\lambda,n_\lambda} \}$
and $\Geul = \bigcup_{\lambda \in \Lambda} \Geul_\lambda$.
Then the elements of $\Geul$ are pairwise relatively prime.

Note that $\Geul \subseteq \Aeul^*$; 
let $\langle \Geul \rangle$ be the subgroup of $\Aeul^*$ generated by $\Geul$
and $\langle F-\lambda : \lambda \in \Lambda \rangle$ the subgroup of 
$\langle \Geul \rangle$ generated by $\setspec{ F-\lambda }{ \lambda \in \Lambda }$.
Then $\langle \Geul \rangle$ and $\langle F-\lambda : \lambda \in \Lambda \rangle$ are free abelian groups of ranks
$|\Geul| = \sum_{\lambda \in \Lambda} n_\lambda$ and $|\Lambda|$ respectively.
Let $\phi : \langle \Geul \rangle \to \Aeul^*/\bk(F)^*$ be the composition 
$\langle \Geul \rangle \hookrightarrow \Aeul^* \xrightarrow{\pi} \Aeul^*/\bk(F)^*$ where $\pi$ is the canonical epimorphism. 
It is easy to see that each element of $\Aeul^*$ has the form $\alpha G$
for some $\alpha \in \bk(F)^*$ and some $G \in A$ where $G$ is a product of elements of $\Geul$.
So $\phi$ is surjective and consequently the abelian group  $\Aeul^*/\bk(F)^*$ is finitely generated.
Since, by \eqref{5923ueip23fnijq3} and \eqref{9823r723txhdi2ed}, $\bk(F)$ is algebraically closed in $W$,
it follows in particular that $\Aeul^*/\bk(F)^*$ 
is torsion-free; so $\Aeul^*/\bk(F)^*$ is a free abelian group of finite rank.
We leave it to the reader to
check that the kernel of $\phi$ is $\langle F-\lambda : \lambda \in \Lambda \rangle$. So
$$
1 \to \langle F-\lambda : \lambda \in \Lambda \rangle \to \langle \Geul \rangle \xrightarrow{\phi} \Aeul^*/\bk(F)^* \to 1
$$
is an exact sequence and it follows that the rank of  $\Aeul^*/\bk(F)^*$ is $|\Geul| - |\Lambda|$, i.e.,
\begin{equation} \label {09832oeuwhsdk}
\textstyle 
\text{$\Aeul^*/\bk(F)^*$ is a free abelian group of rank $\sum_{\lambda \in \bk} (n_\lambda-1)$.}
\end{equation}

Let $R_0, \dots, R_{t-1}$ be the distinct dicriticals of $F$, i.e., 
$$
\proj_\infty(\Aeul/\bk(F)) = \{ R_0, \dots, R_{t-1} \} .
$$
For each $i=0,\dots,t-1$, let $v_i : W^* \to \Integ$ be the valuation of $R_i$.
Since $W/\bk(F)$ has property $(*)$ by \ref{pf9u293jdnaKSFJ}, we may choose,
for each $i \in \{1, \dots, t-1\}$, an element $\xi_i$ of $W \setminus \bk(F)$ satisfying
$\supp( \div \xi_i) = \{ R_0, R_i\}$. Note that $\xi_i$ and $\xi_i^{-1}$ belong to 
$\bigcap \proj_{\text{\rm fin}}(\Aeul/\bk(F)) = \Aeul$, so $\xi_i \in \Aeul^*$.
Let $\langle \xi_1, \dots, \xi_{t-1} \rangle$ be the subgroup of $\Aeul^*$ generated by $\xi_1, \dots, \xi_{t-1}$
and let $\psi : \langle \xi_1, \dots, \xi_{t-1} \rangle \to \Aeul^*/\bk(F)^*$ be the composition
$\langle \xi_1, \dots, \xi_{t-1} \rangle \hookrightarrow \Aeul^* \xrightarrow{\pi} \Aeul^*/\bk(F)^*$.
To complete the proof, it's enough to prove:
\begin{equation} \label {7865e3gaxbk1p902}
\textstyle 
\begin{minipage}[t]{.9\textwidth}
$\langle \xi_1, \dots, \xi_{t-1} \rangle$ is free of rank $t-1$,
$\psi$ is injective and\ \ $\big( \Aeul^*/\bk(F)^* \big) / \image\psi$\ \ is torsion.
\end{minipage}
\end{equation}
Indeed, if this is true then the rank of $\Aeul^*/\bk(F)^*$ is equal to $t-1$,
so the desired equality follows from \eqref{09832oeuwhsdk}.

For each $i=1, \dots, t-1$, let $m_i = v_i(\xi_i) \in \Integ$ and note that $m_i \neq 0$.
Also note that $v_j(\xi_i) = 0$ for all choices of elements $i\neq j$ of $\{1, \dots, t-1\}$.

Suppose that $(k_1, \dots, k_{t-1}) \in \Integ^{t-1}$ is such that $\prod_{i=1}^{t-1} \xi_i^{k_i} \in \ker\psi$.
Then $\prod_{i=1}^{t-1} \xi_i^{k_i} \in \bk(F)^*$, so 
for each $j \in \{1, \dots, t-1\}$ we have
$0 = v_j( \prod_{i=1}^{t-1} \xi_i^{k_i} ) = k_j m_j$, so $k_j=0$.
This proves that 
$\langle \xi_1, \dots, \xi_{t-1} \rangle$ is free of rank $t-1$ and that $\psi$ is injective.

Let $u \in \Aeul^*$.  Choose $N>0$ so that $m_i \mid v_i( u^N )$ for all $i \in \{1, \dots, t-1\}$ and
define $(k_1, \dots, k_{t-1}) \in \Integ^{t-1}$ by $m_i k_i = v_i( u^N )$ for all $i \in \{1, \dots, t-1\}$;
let $\xi = \prod_{i=1}^{t-1} \xi_i^{k_i} \in \Aeul^*$.
Then the element $u^N\xi^{-1}$ of $\Aeul^*$ satisfies $v_i( u^N\xi^{-1} ) = 0$  for all $i \in \{1, \dots, t-1\}$.
We also have $\supp( \div (u^N\xi^{-1}) ) \subseteq \{R_0, \dots, R_{t-1} \}$, because $u^N\xi^{-1} \in \Aeul^*$.
So $\supp( \div (u^N\xi^{-1}) ) \subseteq \{R_0\}$ and hence $\div (u^N\xi^{-1}) = 0$.
Consequently,  $u^N\xi^{-1} \in \bk(F)^*$, so $\pi(u)^N = \psi(\xi)$.
This shows that $\big( \Aeul^*/\bk(F)^* \big) / \image\psi$ is torsion, which completes the
proof of \eqref{7865e3gaxbk1p902}. The theorem is proved.
\end{proof}

\medskip
\noindent{\bf Acknowledgements.}  The author wishes to thank Pierrette Cassou-Nogu\`es for suggesting some
of the questions that are answered in this paper, and Richard Ganong for discussing with him the question
of exotic lines mentioned near the end of Section~\ref{Sec:Definitionsandstatementsofresults}.

%\bibliographystyle{alpha}
%\bibliography{/Users/ddaigle/AAA/articles/bib/dbase}

\end{document}